\documentclass{amsart}

\usepackage{amssymb}
\usepackage{amsmath}
\usepackage{amsthm}
\usepackage[dvips]{graphicx}
\usepackage{color}

\newtheorem{theo}{Theorem}[section]
\newtheorem{lemm}[theo]{Lemma}

\newtheorem{coro}[theo]{Corollary}

\newtheorem{rema}[theo]{Remark}

\newcommand\NN{{\mathbb N}}
\newcommand\RR{{\mathbb R}}

\def\SS{\mathbb {S}}

\def\e{\eqno}
\def\la{\langle}
\def\ra{\rangle}

\let\a=\alpha
\let\b=\beta
\let\d=\delta
\let\e=\epsilon
\let\g=\gamma
\let\lam=\lambda

\let\p=\psi

\let\s=\sigma
\let\ve=\varepsilon
\let\vp=\varphi

\let\D=\Delta

\newcommand\bN{{\mathbb N}}

\newcommand\bR{{\mathbb R}}

\newcommand\bS{{\mathbb S}}

\newcommand\cF{{\mathcal F}}

\newcommand\cS{{\mathcal S}}
\newcommand\cM{{\mathcal M}}

\newcommand\cK{{\mathcal K}}
\newcommand\cG{{\mathcal G}}

\let\wt=\widetilde

\let\dis=\displaystyle

\date{today}
\begin{document}
\title[Measure Valued Solutions]
{Moment classification of infinite energy solutions  to the homogeneous Boltzmann equation
}
\author{Yoshinori Morimoto }
\address{Yoshinori Morimoto, Graduate School of Human and Environmental Studies,
Kyoto University,
%\newline\indent
Kyoto, 606-8501, Japan} \email{morimoto@math.h.kyoto-u.ac.jp}
\author{Shuaikun Wang}
\address{Shuaikun Wang , Department of mathematics, City University of Hong Kong,
Hong Kong, P. R. China
% \newline\indent
% and \newline\indent
% School of Mathematics, Wuhan University 430072,
% Wuhan, P. R. China
} \email{shuaiwang4-c@my.cityu.edu.hk}
\author{Tong Yang}
\address{Tong Yang, Department of mathematics, City University of Hong Kong,
Hong Kong, P. R. China
% \newline\indent
% and \newline\indent
% School of Mathematics, Wuhan University 430072,
% Wuhan, P. R. China
} \email{matyang@cityu.edu.hk}



\date{}

\begin{abstract}
In this paper, we will introduce a precise classification of characteristic
functions in the Fourier space  according to the
moment constraint in the physical space of any order. Based on this, we 
construct measure valued solutions to the homogeneous Boltzmann equation
with the exact moment condition as the initial data.
\end{abstract}
\maketitle
%\tableofcontents

\section{Introduction}\label{s1}

Consider the spatially homogeneous Boltzmann equation,
\begin{equation}\label{bol}
\partial_t f(t,v) %+v\cdot\nabla_x f
=Q(f, f)(t,v),%\quad f(0,v)=f_0(v),
\end{equation}
where $f(t,v)$ is the density distribution of particles with
velocity $v \in \RR^3$ at time $t$. The most interesting and
important part of this equation is the collision operator
given on the
 right hand side that captures the change rate of the density distribution
through the elastic binary collisions:
\[
Q(g, f)(v)=\int_{\RR^3}\int_{\mathbb S^{2}}B\left({v-v_*},\sigma
\right)
 \left\{g(v'_*) f(v')-g(v_*)f(v)\right\}d\sigma dv_*\,,
\]
where for $\sigma \in \SS^2$
$$
v'=\frac{v+v_*}{2}+\frac{|v-v_*|}{2}\sigma,\,\,\, v'_*
=\frac{v+v_*}{2}-\frac{|v-v_*|}{2}\sigma,\,
$$
that follow from the conservation of momentum and energy,
\[ v' + v_*' = v+ v_*, \enskip |v'|^2 + |v_*'|^2 = |v|^2 + |v_*|^2.
\]
% between pre-collisional velocities
% $(v',v_*')$ and
% post- collisional velocities  $(v,v_*)$.

The natural space of the unknown function $f(t,v)$ to the Boltzmann
equation is the space of the probability distribution with suitable
moment constraint that may reflect the boundedness of the momentum
or energy. 

The purpose of this paper is first to give a precise definition of the
space of a probability distribution 
with $\alpha$-order  moment after Fourier transform,
and then construct the measure valued solution to the Boltzmann equation
in such space with the same parameter $\alpha$ for the initial data and the
solution in the setting of Maxwellian type cross-sections.

More precisely, consider \eqref{bol} with initial datum
\begin{equation}\label{initial}
f(0,v) = dF_0\ge 0,
\end{equation}
where $F_0$ is a probability measure.

%As in \cite{morimoto-2012, YM}, we consider
%The non-negative cross section
% $B(z, \sigma)$ depends only on $|z|$ and the scalar product
%$\frac{z}{|z|}\,\cdot\, \sigma$.
Motivated by the inverse power law, assume that the non-negative cross section
$B$  takes the form of
\begin{equation*}
B(|v-v_*|, \cos \theta)=\Phi (|v-v_*|) b(\cos \theta),\,\,\,\,\,
\cos \theta=\frac{v-v_*}{|v-v_*|} \, \cdot\,\sigma\, , \,\,\,
0\leq\theta\leq\frac{\pi}{2},
\end{equation*}
where
\begin{align}%\label{1.2-0}
&\Phi(|z|)=\Phi_\gamma(|z|)= |z|^{\gamma}, \enskip \mbox{for some $\gamma>-3$},   \notag \\
& b(\cos \theta)\theta^{2+2s}\ \rightarrow K\ \
 \mbox{when} \ \ \theta\rightarrow 0+,  \enskip
\mbox{for $0<s<1$ and $K>0$. }\label{1.2}
\end{align}
Throughout this paper, we will only consider the case when
\[
\Phi (|v-v_*|)=1,
\]
that is called the Maxwellian molecule type cross section. In this case,
 the analysis
relies on the good structure of the equation after taking Fourier transform
in $v$ by the Bobylev formula. And the other cases will be pursued by
the authors in the future.

As usual, the range of $\theta$ can be restricted to $[0,\pi/2]$, by replacing $b(\cos\theta)$
 by its ``symmetrized'' version
\[
[ b(\cos \theta)+b(\cos (\pi-\theta))]{\bf 1}_{0\le \theta\le \pi/2}.
\]

%It has been well studied that the angular singularity in the cross section leads
%to the gain of regularity in the solution. The purpose of this paper is to
%show that this still holds for measure valued solutions. Moreover, we
%improve the previous work on existence theory by Cannonne-Karch \cite{Cannone-Karch-1} because we introduce a new classification of the characteristic functions
%so that the moment constraints can be precisely captured in the Fourier
%space.

We can work on the problem with the following slightly more general assumption on the cross section
\begin{align}\label{index-sing}
\exists \alpha_0 \in (0, 2]\enskip\mbox{ such that} \enskip  (\sin \theta/2)^{\alpha_0} b(\cos \theta)
\sin \theta  \in L^1((0, \pi/2]),
\end{align}
which is fulfilled for the function $b$ in \eqref{1.2} if $2s < \alpha_0$.

Denote by $P_\alpha(\RR^3)$, $\alpha \in [0,2]$ the set of  probability measure  $F$ on $\RR^3$, such that
\[
\int_{\RR^3} |v|^\alpha dF(v) < \infty, \,
\]
and moreover when  $1 < \a \le 2$, it requires that
\begin{align}\label{mean}
\int_{\RR^3} v_j dF(v) = 0, \enskip j =1,2,3\,.
\end{align}
%When $\alpha \in (0, 2)$, $\alpha \ne 1$, we

Following Jacob \cite{Jacob} and Cannone-Karch \cite{Cannone-Karch-1}, call the Fourier transform of a probability measure $F \in P_0(\RR^3)$, that is,
\[
\varphi(\xi) = \hat f(\xi) = \cF(F)(\xi) =\int_{\RR^3} e^{-iv\cdot \xi} dF(v),
\]
a characteristic function.

Put
$\cK = \cF(P_0(\RR^3))$.
Inspired by a series of works by Toscani and his co-authors \cite{Carlen-Gabetta-Toscani, Gabetta-Toscani-Wennberg, toscani-villani}, Cannone-Karch defined a subspace
$\cK^\a$ for $\alpha\ge 0$ as follows:
\begin{align}\label{K-al}
\cK^\alpha =\{ \varphi \in \cK\,;\, \|\varphi - 1\|_{\alpha} < \infty\}\,,
\end{align}
where
\begin{align}\label{dis-norm}
\|\varphi - 1\|_{\alpha} = \sup_{\xi \in \RR^3} \frac{|\varphi(\xi) -1|}{|\xi|^\alpha}.
\end{align}
The space $\cK^\alpha $ endowed with the distance
\begin{align}\label{distance}
\|\varphi - \tilde \varphi\|_{\alpha} = \sup_{\xi \in \RR^3} \frac{|\varphi(\xi) -
\tilde \varphi(\xi) |}{|\xi|^\alpha}
\end{align}
is a complete metric space (see Proposition 3.10 of \cite{Cannone-Karch-1}).
It follows that $\cK^\alpha =\{1\}$ for all $\alpha >2$ and the
following embeddings
(Lemma 3.12 of \cite{Cannone-Karch-1}) hold
\[
\{1\} \subset \cK^\alpha \subset \cK^\beta \subset \cK^0 =\cK \enskip \enskip
\mbox{for all $2\ge \alpha \ge \beta \ge 0$}.
\]

With this classification on the characteristic functions, the global
existence of solution in $\cK^\alpha$ was studied in \cite{Cannone-Karch-1}(see also \cite{morimoto-12}).
However,
even though the inclusion $\cF(P_\a(\RR^3)) \subset \cK^\a$ holds (see Lemma 3.15 of \cite{Cannone-Karch-1}),
the space $\cK^\a$ is strictly larger than $\cF(P_\a(\RR^3))$ for $\a \in (0,2)$, in other word, $\cF^{-1}(\cK^\a) \supsetneq P_\a(\RR^3)$.  Indeed,  it is shown (see Remark 3.16 of \cite{Cannone-Karch-1}) that the function
$\varphi_\a (\xi) = e^{-|\xi|^\a}$, with $\a \in (0,2)$, belongs to $\cK^\a$, but $p_\a(v)= \cF^{-1} (\varphi_\a)(v)$ that is
 the density of $\a$-stable symmetric L\'evy process,
is not contained in $P_\a(\RR^3)$.

On the other hand, we remark that  $\cF(P_2(\RR^3)) = \cK^2$. Indeed, this can
be proved by contradiction. 
If there exists a $\varphi(\xi) \in \mathcal{K}^2$ such that  $F = \mathcal{F}^{-1} (\varphi) \notin P_2$, 
then we may assume there exist $ \omega_0 \in \SS^{d-1} \mbox{and} \enskip A >0 $ such that
\[
\int_{\{|\frac{v}{|v|} -\omega_0| < 10^{-10} \} \cap \{|v| \le A\}} |v|^2 dF(v) \ge 100\|1-\vp\|_2\,,\]
from which we have a contradiction because 
\begin{align*}
&\|1-\varphi\|_2 \ge \sup_{\xi} \frac{\mbox{Re}(1-\varphi(\xi))}
{|\xi|^2}\\
&\ge 2 
\int_{\{|\frac{v}{|v|} -\omega_0| < 10^{-10} \} \cap \{|v| \le A\}} 
\frac{\sin^2 \big\{\frac{ |v||\xi|}{2} \big(\frac{v}{|v|} \cdot \frac{\xi}{|\xi|}\big) \big\}}
{|v|^2|\xi|^2} |v|^2 dF(v) \enskip \mbox{for  $\displaystyle \frac{\xi}{|\xi|} = \omega_0$, $\displaystyle|\xi| = \frac{\pi}{A}$}\\
&\ge \frac{2}{\pi^2}
\int_{\{|\frac{v}{|v|} -\omega_0| < 10^{-10} \} \cap \{|v| \le A\}}  \big(\frac{v}{|v|}\cdot \omega_0 \big)^2
 |v|^2 dF(v) > 50\|1-\vp\|_2,
\end{align*}
by using 
\[
\sin z \ge \frac{2z}{\pi} \enskip \mbox{when} \enskip 0 \le  z \le \frac{\pi}{2}\,.
\]

In order to capture the precise moment constraint in the Fourier
space,  another classification on the characteristic functions
was introduced in \cite{MWY-2014} as follows:
\begin{align}\label{M-al}
\cM^\alpha =\{ \varphi \in \cK\,;\, \|\varphi - 1\|_{\cM^\alpha} < \infty\}\,,\enskip \a \in (0,2)\,,
\end{align}
where
\begin{align}\label{integral-norm}
\|\varphi -1\|_{\cM^\a}  = \int_{\RR^3} \frac{|\varphi(\xi)-1 |}{|\xi|^{3+\a} }d\xi\,.
\end{align}
It was shown  in \cite{MWY-2014} that if $\a\in(0,1)\cap(1,2)$, then $\cM^\a=\cF(P_\a)$. However, for the case $\a=1$,  
$\cM^1\subsetneq\cF(P_1)$.

To give a more precise description of the characterization of $P_\a$ of any
order, 
in this paper, we first introduce 
\begin{align}\label{tilde-M-al}
\wt\cM^\alpha =\{ \varphi \in \cK\,;\, \|Re\varphi - 1\|_{\cM^\alpha}+||\vp-1||_\a < \infty\}\,,\enskip \a \in (0,2)\,,
\end{align}
where $Re\vp$ stands for the real part of $\vp(\xi)$. 
Accordingly,  the imaginary part of $\vp(\xi)$ is denoted by $Im\vp$.

For $\varphi, \tilde \varphi \in \wt\cM^\a$, put
\begin{align*}
\|\varphi -\tilde \varphi \|_{\wt\cM^\a}  = \int_{\RR^3} \frac{|Re\varphi(\xi)-Re\tilde \varphi(\xi) |}{|\xi|^{3+\a} }d\xi,
\end{align*}
and, for any $0<\b<\a<2, {0< \e <1}$, we introduce the distance in $\wt\cM^\a$
as
\begin{align}\label{M-dist}
{dis}_{\a,\b,\e} (\varphi ,\tilde \varphi )= \|\varphi - \tilde \varphi\|_{\wt\cM^\a}+ \|\varphi - \tilde \varphi\|_\b+\|\varphi - \tilde \varphi\|_\b^\e.
\end{align}
%in view of $\cM^\a \subset \cF(P_\a(\RR^3)) \subset \cK^\a$ (see Proposition \ref{chara-2}).

With the above preparation,
 the first main result in this paper can be stated as follows.

\begin{theo}\label{complete-1}  If $0 < \b < \alpha <\g\le 2, { 0<\e<1}$, then
%\begin{align}\label{M-dist}
%|||\varphi - \tilde \varphi |||_{\a,\b} = \|\varphi - \tilde \varphi\|_{\cM^\a}   + \|\varphi - \tilde \varphi\|_\b\,.
%\end{align}
%for any $\b \in (0,\a)$.
the space $\wt\cM^\a$ is a complete metric space endowed with the distance $
dis_{\a,\b,\e} (\cdot ,\cdot)$. 
%If $\b, \b'$ are  in $(0,\a]$,  both distances  $dis_{\a,\b} (\cdot, \cdot)$ and $dis_{\a,\b'} (\cdot, \cdot)$ are equivalent
%in the following sense:
%\begin{align}\label{fourier-equiv}
%\mbox{$\forall \varepsilon >0$,  $\exists \delta>0$ s.t. $dis_{\a,\b} (\varphi ,\varphi_0 ) < \delta$  $\Rightarrow$ $dis_{\a,\b'} (\varphi , \varphi_0 )< \varepsilon $.}
%\end{align}
%{\red and moreover if $\a \ne 1$ then this is true for $\b =\a$.}
Moreover, we have
\begin{align}\label{trivial-inclusion}
&\cK^\g \subset \wt\cM^\a\subset \cK^\a\subset\cK^\b, \\
%&\label{inclusion-spa}
%\cM^\a \subset \cF(P_\a(\RR^3)) \, \big (\, \,  \subsetneq \cK^\a \,\, \big) \enskip \mbox{for} \enskip \a \in (0,2)\,,\\
&\label{equivalence-space}
\quad\wt\cM^\a = \cF(P_\a(\RR^3)).
\end{align}
Furthermore, $\displaystyle \lim_{n \rightarrow \infty} dis_{\a,\b,\e}(\varphi_n, \varphi) = 0$, for $\varphi_n,
\varphi
 \in \wt\cM^\a$,
implies
\begin{align}\label{measure-convergence}
&\lim_{n \rightarrow \infty} \int \psi(v) dF_n(v) = \int \psi(v) dF(v)  %\enskip F_n=\cF^{-1}(\varphi_n), F =\cF^{-1}(\varphi)
% \in P_\a(\RR^3),
\mbox{ for any  $\psi \in C(\RR^3)$} \\
&\mbox{\qquad  \qquad \qquad
 satisfying
the growth condition $|\psi(v)| \lesssim \la v \ra^\a$,}\notag
\end{align}
where $F_n = \cF^{-1}(\varphi_n), F = \cF^{-1}(\varphi) \in P_\a(\RR^3)$.
\end{theo}

%{\begin{rema}\label{Wasserstein}
%The weak convergence \eqref{measure-convergence} is equivalent to the one
%given by the Wasserstein distance (see Theorem 7.12 of \cite{villani3}).
%\end{rema}}
\begin{rema}
	If $\a\in(0,1)\cup(1,2)$, then $\wt\cM^\a=\cM^\a$. If $\a=1$ , $\wt\cM^\a\supsetneq\cM^\a$.
\end{rema}

\begin{proof}
Firstly, for $F\in P_\a,\a\in(0,2)$, denote the Fourier transform of $F$ by $\vp$, then $1-Re\vp=\int_v (1-\cos(\xi\cdot v))dF$. Hence
	\begin{align}\label{base-equiv}
	\int\frac{|1-Re\vp|}{|\xi|^{3+\a}}d\xi=
	\iint\frac{1-\cos(\xi\cdot v)}{|\xi|^{3+\a}}dFd\xi=2\int\frac{\sin^2({\zeta\cdot\s}/2)}{|\zeta|^{3+\a}}d\zeta\int|v|^\a dF.
	\end{align}
	It is proved in \cite{Cannone-Karch-1} that $\cF(P_\a)\subset\cK^\a$. This inclusion and \eqref{base-equiv} show $\wt\cM^\a\supset\cF(P^\a)$. To prove $\wt\cM^\a\subset\cF(P^\a)$, it suffices to show:
	$$\mbox{ for any $\vp\in\wt\cM^\a$, if $\a>1$ , then $F=\cF^{-1}(\vp)$ satisfies \eqref{mean}}.$$
	Assume there exists $\vp\in\wt\cM^\a\subset\cK^\a$, such that
	$a=\int vdF\neq0,\,F=\cF^{-1}(\vp)$. Since
	$F(\cdot+a)\in P_\a$, we know $e^{i\xi\cdot a}\vp(\xi)=\vp_a(\xi)=\cF(F(\cdot+a))\in\cK^\a$. Therefore, we have
	\[\begin{split}
	\sup_\xi\frac{|e^{-i\xi\cdot a}-1|}{|\xi|^\a}&=\sup_\xi\frac{|1-e^{i\xi\cdot a}|}{|\xi|^\a}\\
	&\le\sup_\xi\frac{|1-\vp|}{|\xi|^\a}+\sup_\xi\frac{|\vp-e^{i\xi\cdot a}|}{|\xi|^\a}\\
	&=\sup_\xi\frac{|1-\vp|}{|\xi|^\a}+\sup_\xi\frac{|\vp_a-1|}{|\xi|^\a}<\infty.
	\end{split}\]
	 This gives a contradiction to the fact that if $\a>1,a\neq0$, then $e^{-i\xi\cdot a}\notin\cK^\a$.
	 The inclusion (\ref{trivial-inclusion}) follows from
	\begin{align*}
	\int\frac{|1-Re\vp|}{|\xi|^{3+\a}}d\xi
	\le||1-\vp||_\g\int_{|\xi|<1}\frac{1}{|\xi|^{3+\a-\g}}d\xi+2\int_{|\xi|\ge 1}\frac{1}{|\xi|^{3+\a}}d\xi.
	\end{align*}

	 Secondly,  let $\{\vp_n\}_{n=1}^\infty$ be a Cauchy sequence in $\{\wt\cM^\a,dis_{\a,\b,\e}(\cdot,\cdot)\}$. Then, there exists $N>0$, such that
	$$dis_{\a,\b,\e}(\vp_N,\vp_n)<1,\mbox{ for any }n>N.$$
	As $\{\vp_n\}_{n=1}^\infty$ is also a Cauchy sequence in the complete space $\{\cK^\b,||\cdot||_\b\}$, there exists $\vp\in\cK^\b$ such that
	$$
	||\vp_n-\vp||_\b=\sup_\xi\frac{|\vp_n-\vp|}{|\xi|^\b}\to 0,\mbox{ as }n\to\infty.
	$$
	Then, for any $ \d\in(0,1)$,
	$$
	\int_{\d<|\xi|<\d^{-1}}\frac{|Re\vp-Re\vp_n|}{|\xi|^{3+\a}}d\xi
	\le||\vp-\vp_n||_\b\int_{\d<|\xi|<\d^{-1}}\frac{1}{|\xi|^{3+\a-\b}}d\xi\to0,\mbox{ as }n\to\infty.
	$$ 
	Notice that
	\begin{align*}
	&\int_{\d<|\xi|<\d^{-1}}\frac{|Re\vp-1|}{|\xi|^{3+\a}}d\xi\\
	&\leq\int_{\d<|\xi|<\d^{-1}}\frac{|Re\vp-Re\vp_n|}{|\xi|^{3+\a}}d\xi
	+\int_{\d<|\xi|<\d^{-1}}\frac{|Re\vp_n-Re\vp_N|}{|\xi|^{3+\a}}d\xi\\
	&\quad+\int_{\d<|\xi|<\d^{-1}}\frac{|Re\vp_N-1|}{|\xi|^{3+\a}}d\xi.
	\end{align*}
Then by letting $n\to\infty$, we have 
	\begin{align*}
	\int_{\d<|\xi|<\d^{-1}}\frac{|Re\vp-1|}{|\xi|^{3+\a}}d\xi
	\le 1+\int_{\d<|\xi|<\d^{-1}}\frac{|Re\vp_N-1|}{|\xi|^{3+\a}}d\xi
	\le 1+\int\frac{|Re\vp_N-1|}{|\xi|^{3+\a}}d\xi.
	\end{align*}
	Since $\d\in(0,1)$ is arbitrary, we obtain $\vp\in\wt\cM^\a$.

	Finally, suppose that for $F_n, F \in P_\a(\RR^3)$,  we have
	$$
	\varphi_n = \cF(F_n), \varphi = \cF(F) \in \wt\cM^\a,\enskip \mbox{and} \enskip
	\displaystyle \lim_{n \rightarrow \infty} dis_{\a,\b,\e}(\varphi_n, \varphi) = 0\,.
	$$
	Note that for $R >1$
	\[
	\int_{\{|\xi| \le 1/R\}} \frac{|1-Re\varphi_n (\xi)|}{|\xi|^{3+\alpha}} d \xi
	\le \int_{\{|\xi| \le 1/R\}} \frac{|1-Re\varphi(\xi)|}{|\xi|^{3+\alpha}} d \xi  + \|\varphi_n -\varphi\|_{\wt\cM^\a} \,,
	\]
	then it follows from the proof of Proposition 2.2 in \cite{MWY-2014} that for any $\varepsilon_1 >0$ there exist $R >1$ and $N \in \NN$ such that
	\[
	\int_{\{|v| \ge R \}} |v|^\alpha dF_n(v) +  \int_{\{|v| \ge R \}} |v|^\alpha dF(v) < \varepsilon_1\, \enskip \mbox{if $n \ge N$}.
	\]
	This shows \eqref{measure-convergence} because $\varphi_n \rightarrow \varphi$ in $\cS'(\RR^3)$, and hence, 
	$F_n \rightarrow F$ in $\cS'(\RR^3)$.	
\end{proof}

Hence, $\wt{\cM}^\alpha$ represents precisely $P_\alpha$ in the Fourier
space for $\alpha\in (0,2)$. 
It then leads to a question about how to describe the
subspace in $\cK$ corresponding to  $P_\a$ with $\a>2$. For this,
we introduce the following spaces.

For each $n\ge1(n\in\bN),\a\in(0,2]$, denote 
	\begin{align*}
		 \wt P_{2n+\a}(\bR^3)=&\{F \in P_0 ;  \enskip  \frac{(1+|v|^{2})^n F}{\int(1+|v|^{2})^ndF}\in P_\a(\bR^3)\} .
	\end{align*}
	It should be noted that if $\a >1$ then the above definition requires  
	\begin{equation}
		\int_{\bR^3}v_j(1+|v|^{2})^ndF=0,j=1,2,3\,.\label{zero-momen}
	\end{equation}
We  can then characterize $ \wt P_{2n+\a}$ exactly by using
the space $\wt\cM^\a$ obtained above.

\begin{coro}
	Let $n\in\bN,n\ge1$, then we have the following characterization: if $\a\in(0,2)$
	$$\cF( \wt P_{2n+\a})=(1-\D)^{-n}\wt \cM^\a;$$
	If $\a=2$, 
	$$\cF( \wt P_{2n+2})=(1-\D)^{-n}\cK^2,$$
	where $\D$ is the Laplace operator and the space $(1-\D)^{-n}S$ is defined as:
		\begin{align*}
			&\mbox{  $\vp(\xi) \in (1-\D)^{-n} S$, if there exists $\p(\xi)\in C_b(\bR^3)$ }\\
			&\quad \qquad \qquad \mbox{such that $\p=(1-\D)^{n}\vp $ and $\displaystyle 
				\frac{\p(\xi)}{\p(0)} \in S$.}
		\end{align*}
\end{coro}
\begin{proof}
	For any $F\in \wt P_{2n+\a}$, since $\frac{\big( 1+|v|^{2}\big)^n dF}{\int\big(1+|v|^{2}\big)^ndF}\in P_\a(\bR^3)$, we have 
	$$
	\cF\left(\frac{(1+|v|^{2})^n dF}{\int (1+|v|^{2})^n dF}\right)=\frac{(1-\D)^{n}\cF (dF)}{\int (1+|v|^{2})^ndF}\in\wt\cM^\a.
	$$
	Hence, $\cF(dF)\in(1-\D)^{-n}\wt\cM^\a$. Inversely, if $\vp\in(1-\D)^{-n}\wt\cM^\a$, by the definition, there exists $\p \in C_b(\bR^3)$ 
	such that $\p=(1-\D)^n\vp$ and $\p/ \p(0) \in\wt\cM^\a$. Then,
	$$
	P_\a \ni \frac{\cF^{-1}(\p)}{\p(0)}=\frac{\cF^{-1}((1-\D)^n\vp)}
	{\p(0)}=\frac{\big(1+|v|^{2}\big)^n \cF^{-1}(\vp) }{ \int \big(1+|v|^{2}\big)^n  d \cF^{-1}(\vp)(v)}.
	$$
	Moreover, if $\a>1$, 
	$$
	\int_{\bR^3}v_j d \cF^{-1}(\p)(v)=\int_{\bR^3}v_j\big(1+|v|^{2}\big)^nd \cF^{-1}(\vp)(v)=0,\enskip j=1,2,3.
	$$
	This shows $\cF^{-1}(\vp)\in \wt P_{2n+\a}$.
\end{proof}
\begin{rema}
	It is worth to remark that $\cK^2\subsetneq( 1-\D)^{-1}\cK^0$, because 
	the zero moment condition   \eqref{zero-momen} is not
	assumed for the space $\cK^0$.	
\end{rema}
Thanks to the new characterization of $P_\a$ for
any $\alpha\in(0,2)$ by its exact Fourier image $\wt\cM^\a$, we can improve 
the previous results, given in
\cite{Cannone-Karch-1, morimoto-12, MY,MWY-2014},  concerning the existence and the smoothing effect of measure valued
solutions to the Cauchy problem for the
spatially homogeneous Boltzmann equation with the Maxwellian molecule type cross section without angular cutoff.
The results will be stated in the following theorems.

\begin{theo}\label{existence-base-space}
Assume that $b$ satisfies \eqref{index-sing} for some $\a_0 \in (0,2)$ and let  $\a\in(\a_0,2)$.
If $F_0 \in P_\a(\RR^3)$, then there exists a unique measure valued solution $F_t \in C([0,\infty), P_\a(\RR^3))$ to the Cauchy problem
\eqref{bol}-\eqref{initial}, where the continuity with respect to $t$ is in
 the topology defined in  \eqref{measure-convergence}.
\end{theo}

\begin{rema}
	Assume the initial data $F_0\in \wt P_{2n+\a}$ with $\alpha\in (0,2]$.
Since $F_0$ may belong
to $P_\a$ up to the translation when $\alpha >1$, by Theorem \ref{existence-base-space}, we can obtain the corresponding solution $F(t)\in P_\a(\bR^3)$. However, since
$v_j \la v\ra^{2n}$ with $n\ge 1$ is not  a collision invariant, we can not
expect the condition  \eqref{zero-momen} on the initial data
can propagate in time. Hence, the solution in general does not belong to $ \wt P_{2n+\a}$ for $1<\alpha\le 2$.
\end{rema}

\begin{coro}\label{propagate-coro}
Assume that $b$ satisfies \eqref{index-sing},  $n\ge1,n\in\bN,\a\in(0,2]$.
Let $F_t$ be the measure valued solution to the Cauchy problem
\eqref{bol}-\eqref{initial} with respect to the initial data $F_0\in P_0(\bR^3)$  satisfying
$$
\int|v|^{2n+\a}dF_0<\infty.
$$
Then, for any $T>0$, there exists $C>0$ such that
\begin{align}\label{propagate}
\int|v|^{2n+\a}dF_t\le Ce^{CT}\int|v|^{2n+\a}dF_0,
\end{align}
for any $t\in[0,T]$.
\end{coro}
\begin{proof}
We first consider the case with angular
cutoff cross section 
and use a modified weighted moment as follows.
For $\delta >0$, put
\[W_\delta(v) = \frac{\la v \ra^{2n+\a}}{1+\delta \la v\ra^{2n+\a}}\,.
		\]
Since the function $x/(1+\delta x)$ is increasing in $[1,\infty]$ and $|v'| \le |v| +|v_*|$,  we have
\begin{align*}
W_{\delta}(v') &\lesssim \frac{\la v\ra^{2n+\a} + \la v_*\ra^{2n+\a}}{1+\delta (\la v\ra^{2n+\a} + \la v_*\ra^{2n+\a})}\\
&
= \frac{\la v\ra^{2n+\a} }{1+\delta (\la v\ra^{2n+\a} + \la v_*\ra^{2n+\a})}
+\frac{ \la v_*\ra^{2n+\a}}{1+\delta (\la v\ra^{2n+\a} + \la v_*\ra^{2n+\a})}\\
&\le \frac{\la v\ra^{2n+\a} }{1+\delta \la v\ra^{2n+\a} }
+\frac{ \la v_*\ra^{2n+\a}}{1+\delta \la v_*\ra^{2n+\a}} = W_{\delta}(v)+W_{\delta}(v_*)\,.
\end{align*}
Therefore,
$|W_{\delta}(v') - W_{\delta}(v)| \lesssim W_{\delta}(v)+W_{\delta}(v_*) $. Set
 $b_m(\cos \theta) = \min \{b(\cos \theta), m \} $ and let $f^m(t,v)$ be the
unique solution of the corresponding Cauchy problem. For the simplicity of the notations, we consider the case
where
$F_0 $ and $F^m_t$ have density functions $f_0(v)$ and $ f^m(t,v)$ respectively.
 The general case can
be considered similarly. Then
we have
\begin{align*}
\frac{d }{dt} \int f^m(t,v)W_{\delta}(v) dv \lesssim
\Big(\int b_m d\sigma \Big) \Big(\int f^m(t,v)W_{\delta}(v) dv \Big) \Big(\int f^m(t,v_*)dv_* \Big)\,,
\end{align*}
which yields
\[
\int f^m(t,v)W_{\delta}(v) dv \le C_m e^{C_m t} \int f_0(v)W_{\delta}(v) dv.
\]
Taking the limit $\delta \rightarrow +0$, we have $f^m(t,v) \in L^1_{2n+\a}$.\\
		
In order to overcome the angular singularity, we need a  precise formula used in the Povzner inequality, cf. \cite{mischler-wennberg}. 
To be self-contained, we derive it
as follows.  Since $\sigma \in \SS^2$, it can be  written as
\begin{align*}
\sigma %=\frac{v-v_*}{|v-v_*|}\cos\theta + \sin\theta (\mathbf{h}\cos\theta + %\mathbf{i} \sin\varphi)
=\mathbf{k}\cos\theta+\sin\theta(\mathbf{h}\cos\varphi+ \mathbf{i}\sin\varphi),
\enskip \theta \in [0,\pi), \varphi \in [-\pi,\pi)\,,
\end{align*}
by an orthogonal basis in $\RR^3$, 
\begin{align*}
\mathbf{k}=\frac{v-v_*}{|v-v_*|},\ 
\mathbf{i}=\frac{v\times v_*}{|v\times v_*|},\ 
\mathbf{h}=\mathbf{i} \times \mathbf{k}= \frac{
\big((v-v_*)\cdot v\big) v_* -\big((v-v_*)\cdot v_*\big) v}{|v-v_*| |v \times v_*|}.
\end{align*}
It follows from $(v+v_*)\perp \mathbf{i}$ and the definition of $\mathbf{h}$ that
\begin{align*}
|v'|^2=&|\frac{v+v_*}{2}|^2+|\frac{v-v_*}{2}|^2+ \frac{|v-v_*|}{2}(v+v_*)\cdot \sigma\\
\quad =&\frac{1}{4}(2|v|^2+2|v_*|^2)+\frac{|v-v_*|}{2}\Big((v+v_*)\cdot
(\cos\theta \mathbf{k}+\sin\theta\cos\varphi \mathbf{h}) \Big) \\
\quad =&\frac{1}{2} (|v|^2+|v_*|^2)+\frac{\cos\theta}{2}(|v|^2-|v_*|^2)\\
&\ +\frac{\sin\theta\cos\varphi}{2|v \times v_*|}\Big \{ (v+v_*)
\cdot \Big(\big((v-v_*)\cdot v\big) v_* -\big((v-v_*)\cdot v_*\big) v \Big)\Big\} \\
=&\frac{|v|^2(1+\cos\theta) }{2}+\frac{|v_*|^2(1-\cos\theta)}{2}+|v||v_*|\sin\alpha\sin\theta\cos\varphi \,,
\end{align*}
where $\alpha$ is the angle between $v$ and $v_*$.  Therefore, we have
\begin{align}\label{strong-case}
|v'|^2 &= |v|^2 \cos^2 \frac{\theta}{2} +|v_*|^2 \sin^2 \frac{\theta}{2} +|v \times v_*|\sin\theta\cos\varphi\\
&= Y(\theta) + Z(\theta) \cos \varphi\,.\notag
\end{align}
Similarly, we have
\begin{align}\label{strong-case-2}
|v'_*|^2 &= |v_*|^2 \cos^2 \frac{\theta}{2} +|v|^2 \sin^2 \frac{\theta}{2} -|v \times v_*|\sin\theta\cos\varphi\\
&= Y(\pi-\theta) - Z(\theta) \cos \varphi\,.\notag
\end{align}
If $\Psi(x) = \Psi_{2n+\alpha}(x) = (1+x)^{n+\alpha/2}$, then it follows from the change of variables that
\begin{align*}
&\frac{d }{dt} \int f^m(t,v)\la v \ra^{2n+\a}  dv =  \frac{1}{2}\iint  f^m(t,v) f^m(t,v_*) K(v,v_*)
dvdv_*,
\end{align*}
where
\begin{align*}
 K(v,v_*)&= \int_{\SS^2} b_m  
\big\{\Psi(|v'|^2) + \Psi(|v'_*|^2) -\Psi(|v|^2) -\Psi(|v_*|^2)\big\}d\sigma\\
&=2 \int_0^\pi \int_0^\pi b_m(\cos \theta) \big\{\Psi(|v'|^2) + \Psi(|v'_*|^2) -\Psi(|v|^2) -\Psi(|v_*|^2)\big\}\sin \theta
d\theta d \varphi\,.
\end{align*}
Note that
\begin{align*}
&\int^{\pi}_0 \Psi (Y(\theta )+ Z(\theta )\cos\varphi ) \ d\varphi\\
&=(\int^{\frac{\pi}{2}}_0 + \int^{\pi}_{\frac{\pi}{2}}) \ \Psi (Y(\theta ) + Z(\theta ) \cos\varphi ) \ d\varphi\\
&=\int^{\frac{\pi}{2}}_0 \{ \Psi (Y(\theta )+ Z(\theta )\cos\varphi )+\Psi (Y(\theta )- Z(\theta )\cos\varphi )-2\Psi (Y(\theta )) \} \ d\varphi +\pi \Psi (Y(\theta )),
\end{align*}
and by using integration by parts twice, we have
\begin{align*}
&\int_0^\pi \Psi(|v'|^2) d \varphi = \int^{\pi}_0 \Psi (Y(\theta )+ Z(\theta )\cos\varphi ) \ d\varphi \\
&=\pi \Psi (Y) +[\varphi \{ \Psi (Y+ Z\cos\varphi )+\Psi ( Y- Z \cos\varphi )-2\Psi (Y) \} ]^{\frac{\pi}{2}}_0\\
&\quad -\int^{\frac{\pi}{2}}_0 \varphi \{ {\Psi}'(Y+Z \cos\varphi ) - {\Psi}'(Y- Z\cos\varphi ) \} (- Z\sin\varphi )\ d\varphi \\
&=\pi \Psi (Y) + \int^{\frac{\pi}{2}}_0  Z \varphi \sin\varphi  ({\Psi}'(Y + Z\cos\varphi )-{\Psi}'(Y- Z \cos\varphi ))\ d\varphi \\
&=\pi \Psi (Y) +Z[(\sin\varphi - \varphi \cos\varphi ) \{ {\Psi}'(Y+ Z \cos\varphi ) - {\Psi}'(Y-Z \cos\varphi )\} ]^{\frac{\pi}{2}}_0\\
&\quad + Z^2\int^{\frac{\pi}{2}}_0 (\sin\varphi - \varphi \cos\varphi ) \{ {\Psi}''(Y+ Z \cos\varphi ) +{\Psi}''(Y- Z \cos\varphi ) \} \sin\varphi \ d\varphi \\
&=\pi \Psi (Y(\theta))   +Z^2\int^{\frac{\pi}{2}}_0 (\sin\varphi - \varphi \cos\varphi )\sin\varphi \\
& \qquad \times  \{ ({\Psi}''(Y(\theta)+ Z\cos\varphi ) +{\Psi}''(Y(\theta)-Z \cos\varphi ) \} \ d\varphi .
\end{align*}
Similarly, we obtain
\begin{align*}
&\int_0^\pi \Psi(|v'_*|^2) d \varphi 
=\pi \Psi (Y(\pi-\theta))   +Z^2\int^{\frac{\pi}{2}}_0 (\sin\varphi - \varphi \cos\varphi )\sin\varphi \\
& \qquad \qquad \times  \{ ({\Psi}''(Y(\pi-\theta)+ Z\cos\varphi ) +{\Psi}''(Y(\pi-\theta)-Z \cos\varphi ) \} \ d\varphi .
\end{align*}
In view of these formula, we consider $K(v,v_*)$ by
dividing it into two parts as follows:
\[
K(v,v_*) = -H(v,v_*) + G(v,v_*).
\]
For the first part, we have
\begin{align*}
&-H(v,v_*)=
2\pi \int_0^\pi b_m(\cos \theta) \\
&\qquad \times \big\{\Psi(Y(\theta) )+ \Psi(Y(\pi -\theta)) - (
 \cos^2\frac{\theta}{2} + \sin^2 \frac{\theta}{2} )
 \big(\Psi(|v|^2) +\Psi(|v_*|^2)\big)\big\} d \theta\\
& = 2\pi \int_0^\pi b_m(\cos \theta) \Big [\{ \Psi(|v|^2 \cos^2 \frac{\theta}{2} +|v_*|^2 \sin^2 \frac{\theta}{2}) - 
\cos^2 \frac{\theta}{2} \Psi(|v|^2) - \sin^2 \frac{\theta}{2}\Psi(|v_*|^2 )\} \\
&\qquad + \{ \Psi(|v_*|^2 \cos^2 \frac{\theta}{2} +|v|^2 \sin^2 \frac{\theta}{2}) - \cos^2 \frac{\theta}{2} \Psi(|v_*|^2 )- 
 \sin^2 \frac{\theta}{2} \Psi(|v|^2) \} \Big ]d\theta \le 0,
\end{align*}
where we have used the fact that $\Psi$ is concave.   On the other hand, if $Z_0 = Z(\theta)/(1+Y(\theta)) \in [0,1]$, then
\begin{align*}
&Z^2\int^{\frac{\pi}{2}}_0 (\sin\varphi - \varphi \cos\varphi )\sin \varphi \{ {\Psi}''(Y+ Z \cos\varphi ) +{\Psi}''(Y- Z \cos\varphi ) \} \ d\varphi \\
&\lesssim 
Z^2(1+ Y)^{n-2+\a/2} \int_0^{\pi/2} \varphi^3 \{(1 + Z_0 \cos \varphi)^{n-2+\a/2} +  (1 - Z_0 \cos \varphi)^{n-2+\a/2}\}d\varphi\\
&\left\{
\begin{array}{ll}
\lesssim Z^2 \lesssim |v|^2|v_*|^2 \theta^{2}, \enskip &\mbox{if}\enskip n =1;\\\\
\lesssim Z^2(1+ Y)^{n-2+\a/2} \lesssim (1+ |v|^2 + |v_*|^2)^{n+\a/2} \theta^2,\enskip &\mbox{if} \enskip n \ge 2\,.
\end{array}
\right .
\end{align*}
Therefore,  if $n =1$ then $G(v,v_*) \lesssim |v|^2 |v_*|^2$, which implies
\[
\frac{d }{dt} \int f^m(t,v)\la v \ra^{2n+\a}  dv  \le C_0\Big( \int |v|^2 f^m(t,v) dv\Big)^2 = C_0\Big(\int |v|^2 f_0(v) dv\Big)^2,
\]
where the  constant $C_0$ is independent of $m$.
When $n \ge 2$,  there exists another constant  $C_1 >0$ independent of $m$ such that
\begin{align*}
G(v,v_*) \le C_1\big (\la v \ra^{2n+\alpha} + \la v_* \ra^{2n+\alpha}\big )\,,
\end{align*}
which implies
\begin{align*}
\frac{d }{dt} \int f^m(t,v)\la v \ra^{2n+\a}  dv 
 \le 4 C_1  \int f^m(t,v) \la v \ra^{2n+\a} dv\,.
\end{align*}

Finally, take a cutoff function $\chi(v)\in C_0^\infty(\bR^3)$ satisfying  $0\le\chi(v)\le1$ and $\chi=1$ on $\{|v|\le1\}$. Then, for any $R>0$, we have
$$\int f^m(t,v)\la v \ra^{2n+\a} \chi\left(\frac{v}{R}\right)   dv \le C e^{Ct} \int f_0(v)\la v \ra^{2n+\a}  dv. $$
Since $f^m(t,v)\rightarrow f(t,v)$ in $\cS'(\RR^3_v)$ and $\la v \ra^{2n+\a} \chi\left(\frac{v}{R}\right)  \in \cS$, we get
$$\int f(t,v)\la v \ra^{2n+\a} \chi\left(\frac{v}{R}\right)     dv \le C e^{Ct} \int f_0(v)\la v \ra^{2n+\a}  dv. $$
Letting $R \rightarrow  \infty$ shows for any $T>0$, $F_t\in L^\infty([0,T],P_{2n+\a})$.
And this completes the proof of Corollary  \ref{propagate-coro}.
\end{proof}
	
The proof of the  Theorem 1.5 will be given in the Fourier space. In fact,
by letting $\varphi(t,\xi) =\cF(F_t)$ and $\varphi_0= \cF(F_0)$,
it follows from the Bobylev formula that
the Cauchy problem \eqref{bol}-\eqref{initial} is reduced to
\begin{equation}\label{c-p-fourier}
\left \{
\begin{array}{l}\dis \partial_t \varphi(t,\xi)
=\int_{\SS^2}b\left(\frac{\xi \cdot \sigma}{|\xi|}\right) \Big( \varphi(t,\xi^+)\varphi(t, \xi^-) - \varphi(t, \xi)
\varphi(t,0)\Big) d\sigma, \\\\
\dis \varphi(0,\xi)=\varphi_0(\xi), \enskip \mbox{where} \enskip
\dis \xi^\pm = \frac{\xi}{2} \pm \frac{|\xi|}{2} \sigma\,.
\end{array}
\right.
\end{equation}

By  Theorem \ref{complete-1}, to prove Theorem \ref{existence-base-space}
it suffices to show
\begin{theo}{\label{fourier-space}}
Assume that $b$ satisfies \eqref{index-sing} for some $\a_0 \in (0,2)$. Let    $2>\a>\b>\max\{\a_0,\a/2\}$ and $\e\in(0,1-\frac{\a_0}\b]$.
If the initial datum $\varphi_0$ belongs to $\wt\cM^\a$, then there exists a unique classical solution $\varphi(t,\xi) \in C([0,\infty), \wt\cM^\a)$
to the Cauchy problem \eqref{c-p-fourier} satisfying that, for all $t,s\in[0,T]$,
\begin{align}\label{fourier-continuity-1}
&||\vp(t)-\vp(s)||_\b\lesssim e^{\lam_\b\max\{s,t\}}|t-s|,\\
&||\vp(t)-\vp(s)||_{\wt\cM^\a}\lesssim C(t,s)|t-s|,\label{fourier-continuity-2}
\end{align}
where $C(t,s)=e^{\lam_\a \max\{s,t\}}||\vp_0-1||_{\wt\cM^\a}+(e^{\lam_\b \max\{s,t\}}||1-\vp_0||_\b
+1)^2$, and
\begin{equation}\label{05}
  \lam_i= 2 \pi \int_0^{\pi/2} b\left( \cos \theta \right)
             \left(\cos^i\frac\theta2+\sin^i\frac\theta2-1\right) \sin \theta d \theta >0 \,,i=\a,\b. %d\s
\end{equation}
Furthermore, if $\vp(t,\xi),\tilde \vp(t,\xi)\in C([0,\infty),\wt\cM^\a)$ are two solutions to the Cauchy problem \eqref{c-p-fourier}
with initial data $\vp_0,\tilde \vp_0\in\wt\cM^\a$, respectively, then for any $t>0$, the following two stability estimates hold
\begin{equation}
||\vp(t)-\tilde \vp(t)||_{\a} \leq e^{\lam_\a t}||\vp_0-\tilde \vp_0||_{\a}\,, \label{alpha-stability}
\end{equation}
\begin{equation}\label{fourier-stability}
  ||\vp(t)-\tilde \vp(t)||_{\wt\cM^\a}
 \lesssim e^{\lam_\a t}||\vp_0-\tilde \vp_0||_{\wt\cM^\a}+\frac{e^{2\lam_\b t}-e^{\lam_\a t}}{2\lam_\b-\lam_\a}A+\frac{e^{\lam_\b t}-e^{\lam_\a t}}{\lam_\b-\lam_\a}B\,,
\end{equation}
where 
\begin{align*}
&A=\max\{||1-\vp_0||_\b,||1-\tilde\vp_0||_\b\}\cdot ||\vp_0-\tilde\vp_0||_\b\,,\\
&B=||\vp_0-\tilde\vp_0||_\b+||1-\tilde\vp_0||_\b^{1-\ve}||\vp_0-\tilde\vp_0||_\b^\ve.
\end{align*}
\end{theo}
\begin{rema}\label{rem-1234}
Since $\vp_0,\tilde \vp_0\in\wt\cM^\a \subset \cK^\a$, the stability estimate \eqref{alpha-stability} is nothing but (13) of \cite{morimoto-12}.
\end{rema}

Finally, we give the following corollary about the regularity of the
solutions.

\begin{coro}\label{smoothing}
Let $b(\cos \theta)$ satisfy \eqref{index-sing}
and let $\alpha \in (\alpha_0, 2]$.
If $F_0 \in  P_\alpha(\RR^3)$ is not a single Dirac mass and  $f(t,v)$ is the unique solution 
in $C([0,\infty),  P_\alpha(\RR^3))$ to the Cauchy problem \eqref{bol}-\eqref{initial}, then
$f(t, \cdot )$ belongs to $L^1_\a (\RR^3) \cap H^\infty(\RR^3) $ for any $t >0$.
\end{coro}

\begin{rema} The case except $\a \ne 1$ in the above corollary
was already proved in Theorem 1.8 in \cite{MWY-2014}.
 The  newly defined  space $\wt \cM_\a$ in this paper 
fills the gap in the case when  $\a =1$. 
\end{rema}

This ends the introduction and the proofs of Theorems \ref{existence-base-space}
and \ref{fourier-space} will be given in the next section.

\section{Proof of  Theorem  \ref{fourier-space}}\label{s-34}

This section concerns with the existence of measure
valued solutions in the new classification of the characteristic functions.
We only need to prove  Theorem  \ref{fourier-space} because Theorem \ref{existence-base-space} will then follow by using Theorem \ref{complete-1}.

Let $b(\cdot)$ satisfy \eqref{index-sing} . As usual, the
existence for non-cutoff cross section is based on the cutoff approximations $b_n(\cdot)=\min\{b(\cdot),n\}$.
Following the previous works \cite{Cannone-Karch-1, morimoto-12, MY,MWY-2014}, 
define the following constants for $\a\in[\a_0,2)$:
\begin{align}
\g_\a^n&=\int_{\bS^2}b_n\left(\frac{\xi\cdot\s}{|\xi|}\right)\left(\sin^\a\frac\theta2+\cos^\a\frac\theta2\right)d\s>0,\nonumber\\
\lam_\a^n&=\int_{\bS^2}b_n\left(\frac{\xi\cdot\s}{|\xi|}\right)\left(\sin^\a\frac\theta2+\cos^\a\frac\theta2-1\right)d\s
=\g_\a^n-\g_2^n>0,\nonumber\\
\lam_\a&=\int_{\bS^2}b\left(\frac{\xi\cdot\s}{|\xi|}\right)\left(\sin^\a\frac\theta2+\cos^\a\frac\theta2-1\right)d\s>0.\label{lam-def}
\end{align}
Note that $\lam_\a$ is finite and independent of $\xi$; furthermore, $\{\lam_\a^n\}_{n=1}^\infty$ converges monotonically to $\lam_\a$.

%The existence and stability with cutoff cross section can be stated as follows.

%\begin{prop}\label{cutoff-existence}
%Let $\a\in (0,2)$ and $b$ satisfy \eqref{index-sing}. For every initial datum $\varphi_0 \in \cK^\a$, there exists
%a unique classical solution to the Cauchy problem \eqref{c-p-fourier} satisfying $\varphi(t) \in C([0, \infty), \cK^\a)$.
%Furthermore, if $\vp(t,\xi),\tilde \vp(t,\xi)\in C([0,\infty),\cK^\a)$ are two solutions to the Cauchy problem \eqref{c-p-fourier}
%with initial data $\vp_0,\tilde \vp_0\in\cK^\a$, respectively, then for any $t>0$ we have
%\begin{align}\label{cutoff-sta}
% & ||\vp(t)-\tilde \vp(t)||_{\a} \leq e^{\lam_\a t}||\vp_0-\tilde \vp_0||_{\a}\,,
%\end{align}
%where $\lambda_\a$ is defined by \eqref{lam-def}.
%\end{prop}
Following Subsection 4.2 of \cite{Cannone-Karch-1},  consider
 the nonlinear operator,
$$
\cG_n(\vp)(\xi)\equiv\int_{\bS^2} b_n \Big (\frac{\xi\cdot\sigma}{|\xi|}\Big )\vp(\xi^+)\vp(\xi^-)d\sigma.
$$
Then,  problem \eqref{c-p-fourier} can be formulated by
\begin{align}\label{integral-eq}
\vp(\xi,t)=\vp_0(\xi)e^{-\g_2^nt}+\int_0^t e^{-\g_2^n(t-\tau)}\cG_n(\vp(\cdot,\tau))(\xi)d\tau.
\end{align}

For the nonlinear operator $\cG_n(\cdot)$, we have the following estimate.

\begin{lemm}\label{lemm01}
Assume $b(\cdot)$ satisfies \eqref{index-sing}. Let $b_n=\min\{b,n\}$. Let $\max\{\a_0,\a/2\}<\b<\a<2$ and $\e\in(0,1-\frac{\a_0}\b]$. Then there exists $C>0$ 
 independent of $n$  such that, for all $\vp,\tilde \vp\in\wt\cM^\a$, we have
\begin{align}\label{g-operator-1}
||\cG_n(\vp)-\cG_n(\tilde\vp)||_\b\le \g_\b^n||\vp-\tilde\vp||_\b,
\end{align}
\begin{align}
||\cG_n(\vp)-\cG_n(\tilde\vp)||_{\wt\cM_\a}
\leq&\g_\a^n||\vp-\tilde\vp||_{\wt\cM_\a}+C\max\{||1-\vp||_\b,||1-\tilde\vp||_\b\}\cdot ||\vp-\tilde\vp||_\b \nonumber\\
&+C||\vp-\tilde\vp||_\b+C||1-\tilde\vp||_\b^{1-\ve}||\vp-\tilde\vp||_\b^\ve,\label{g-operator-2}
\end{align}
for any $n\ge1(n\in\bN)$.
In particular, if $\tilde\vp\equiv1$,
\[
\begin{split}
\int_{\bR^3}\frac{|Re\cG_n(\vp)-\g_2^n|}{|\xi|^{3+\a}}d\xi
\le\g_\a^n||\vp-1||_{\wt\cM_\a}+C||\vp-1||_\b^2
+C||\vp-1||_\b.
\end{split}
\]
\end{lemm}
\begin{proof}
	The proof of \eqref{g-operator-1} can be found in \cite{Cannone-Karch-1}. We only need to prove other estimates in the lemma. Firstly, note that
\[
\begin{split}
    ||\cG_n(\vp)-\cG_n(\tilde\vp)||_{\wt\cM_\a}&=
	\int_{\bR^3}\frac{|Re\cG_n(\vp)-Re\cG_n(\tilde\vp)|}{|\xi|^{3+\a}}d\xi \\
	& \leq \int_{\bR^3}\int_{\bS^2}b_n(\cdot)\left(\frac{
		|Re\vp^+Re\vp^--Re\tilde\vp^+R\tilde\vp^-|}{|\xi|^{3+\a}}\right.\\
	&\qquad\qquad\left.+
	\frac{|Im\vp^+Im\vp^--Im\tilde\vp^+Im\tilde\vp^-|}{|\xi|^{3+\a}}\right) d\s d\xi\\
	& =I_1+I_2.
\end{split}
\]
For $I_1$, we can apply the Lemma 3.2 in \cite{MWY-2014}.  That is,
$$
I_1=\int_{\bR^3}\int_{\bS^2}b_n(\cdot)\frac{
	|Re\vp^+Re\vp^--Re\tilde\vp^+Re\tilde\vp^-|}{|\xi|^{3+\a}} d\s d\xi\leq\g_\a^n\int_{\bR^3}\frac{|Re\vp-Re\tilde\vp|}{|\xi|^{3+\a}}d\xi.
$$
For $I_2$, we divide it into two parts: $I_{2,1}$ for $|\xi|<1$ and $I_{2,2}$
for $|\xi|>1$. Then, 
\[
\begin{split}
I_{2,1}
	&=\int_{|\xi|<1}\int_{\bS^2}b_n(\cdot)\frac{|Im\vp^+\cdot Im(\vp^--\tilde\vp^-)+Im\tilde\vp^-\cdot Im(\vp^+-\tilde\vp^+)|}{|\xi|^{3+\a}} d\s d\xi\\
	&\leq2\int_{|\xi|<1}\frac{d\xi}{|\xi|^{3+\a-2\b}}\int_{\bS^2} b(\cdot)\sin^\b\frac\theta2\cos^\b\frac\theta2d\s
	\cdot ||\vp-\tilde\vp||_\b\\
	&\qquad\cdot\max\{||1-\vp||_\b,||1-\tilde\vp||_\b\}\\
	&= C ||\vp-\tilde\vp||_\b\cdot\max\{||1-\vp||_\b,||1-\tilde\vp||_\b\}.
	\end{split}
	\]
By using 
	\[
	\begin{split}
	\int_{|\xi|>1}\frac{|Im\vp^+||Im(\vp^--\tilde\vp^-)|}{|\xi|^{3+\a}}d\xi
	&\leq\int_{|\xi|>1}\frac{\sin^\b\frac\theta2}{|\xi|^{3+\a-\b}}\frac{|Im(\vp^--\tilde\vp^-)|}{|\xi^-|^{\b}}d\xi\\
	&\leq C||\vp-\tilde\vp||_\b\cdot\sin^{\b}\frac\theta2,
	\end{split}
	\]
	and
	\[
	\begin{split}
	&\int_{|\xi|>1}\frac{|Im\tilde\vp^-||Im(\vp^+-\tilde\vp^+)|}{|\xi|^{3+\a}}d\xi\\
	&\leq\left[\int_{|\xi|>1}\frac{|Im\tilde\vp^-||Im(\vp^+-\tilde\vp^+)|}{|\xi|^{3+\a}}d\xi\right]^{1-\ve}
	\left[\int_{|\xi|>1}\frac{|Im\tilde\vp^-||Im(\vp^+-\tilde\vp^+)|}{|\xi|^{3+\a}}d\xi\right]^\ve\\
	&\leq\left[\int_{|\xi|>1}\frac{2\sin^\b\frac\theta2}{|\xi|^{3+\a-\b}}\cdot\frac{|Im\tilde\vp^-|}{|\xi^-|^\b}d\xi\right]^{1-\ve}
	\left[\int_{|\xi|>1}\frac{\cos^\b\frac\theta2}{|\xi|^{3+\a-\b}}\cdot\frac{|Im(\vp^+-\tilde\vp^+)|}{|\xi^+|^\b}d\xi\right]^\ve\\
	&\leq 2\sin^{\b(1-\ve)}\frac\theta2\cos^{\b\ve}\frac{\theta}{2}\int_{|\xi|>1}\frac1{|\xi|^{3+\a-\b}}d\xi ||1-\tilde\vp||_\b^{1-\ve}||\vp-\tilde{\vp}||_\b^\ve,
	\end{split}
	\]
	%{\color{red}{
	%		or if we let $\b=\a$,
	%		\[
	%		\begin{split}
	%		\int_{|\xi|>1}\frac{|Im\vp^+||Im\vp^--Im\tilde\vp^-|}{|\xi|^{3+\a}}d\xi
	%		&\leq\left(\int_{|\xi|>1}\frac{|Im\tilde\vp^+||Im\vp^--Im\tilde\vp^-|}{|\xi|^{3+\a-\ve}}d\xi\right)^\ve
	%		\left(\int_{|\xi|>1}\frac{|Im\tilde\vp^+||Im\vp^--Im\tilde\vp^-|}{|\xi|^{3+\a+\ve}}d\xi\right)^{1-\ve}\\ 
	%		&\leq\left(\int_{|\xi|>1}\frac{2\cos^\a\frac\theta2}{|\xi|^{3-\ve}}\cdot\frac{|Im\tilde\vp^+|}{|\xi^+|^\b}d\xi\right)^{\ve}
	%		\left(\int_{|\xi|>1}\frac{\cos^\b\frac\theta2}{|\xi|^{d+\a-\b}}\cdot\frac{|Im\vp^--Im\tilde\vp^-|}{|\xi^-|^\b}d\xi\right)^{1-\ve}\\
	%		&\leq 2\sin^{\b(1-\ve)}\frac\theta2\cos^{\b\ve}\frac{\theta}{2}\int_{|\xi|>1}\frac1{|\xi|^{d+\a-\b}}d\xi ||1-\tilde\vp||_\b^{1-\ve}||\vp-\tilde{\vp}||_\b^\ve.
	%		\end{split}
	%		\]
	%	}}\\
	we obtain 
	\begin{align}
	I_{2,2}
	\leq &C\int b(\cdot)\sin^\b\frac\theta2d\s||\vp-\tilde\vp||_\b\\
	&+2C\int b(\cdot)\sin^{\b(1-\ve)}\frac\theta2\cos^{\b\ve}\frac\theta2d\s||1-\tilde\vp||_\b^{1-\ve}||\vp-\tilde\vp||_\b^\ve.\nonumber
	\end{align}
	Therefore,
	\[
	\begin{split}
	\int_{\bR^3}\frac{|Re\cG_n(\vp)-Re\cG_n(\tilde\vp)|}{|\xi|^{3+\a}}d\xi
	\leq&\g_\a^n\int_{\bR^3}\frac{|Re\vp-Re\tilde\vp|}{|\xi|^{3+\a}}d\xi\\
	&+C\max\{||1-\vp||_\b,||1-\tilde\vp||_\b\}\cdot ||\vp-\tilde\vp||_\b\\
	&+C||\vp-\tilde\vp||_\b+C||1-\tilde\vp||_\b^{1-\ve}||\vp-\tilde\vp||_\b^\ve.\\
	\end{split}
	\]
And this completes the proof of the lemma.
\end{proof}

\subsection{Existence under the cutoff assumption}
We are now ready to prove the existence of the solution under the cutoff assumption.
%Banach contraction principle:
The solution to \eqref{c-p-fourier} with $b$ replaced by $b_n$ can be obtained as a fixed point of \eqref{integral-eq}  to
the nonlinear operator
$$
\cF_n(\vp)(t,\xi)\equiv\vp_0(\xi)e^{-\g_2^nt}+\int_0^te^{-\g_2^n(t-\tau)}\cG_n(\vp(\tau))(\xi)d\tau,
$$
for a fixed $\varphi_0 \in \wt\cM^\a$.
 For a fixed $T >0$  to be determined later, denote
\begin{align*}
X_n=\{\vp(t,\xi)\in  C([0,T],\wt\cM^\a):\,\, &
\vp(0,\xi) = \vp_0(\xi), \\
\forall t\in[0,T], &||1-\vp(t,\cdot)||_\b\le e^{\lam_\b^n t}||1-\vp(0,\cdot)||_\b \}
\end{align*}
supplemented with the metric
\begin{align*}
||\vp-\tilde\vp||_{X_n}=\sup_{t \in[0,T]}   dis_{\a,\b,\e}( \vp(t), \tilde\vp(t)),
\end{align*}
for  $\vp,\tilde\vp\in X_n$.

Firstly, the following lemma gives a local in time
existence of solution in $X_n$.

\begin{lemm}\label{contr-lemm}
	There exists $T>0$, such that $\cF_n:{X_n}\to {X_n}$ is a contraction mapping.
\end{lemm}
\begin{proof}
	Firstly, we prove that $\cF_n$ maps ${X_n}$ into itself. By Lemma \ref{lemm01}, for any $\vp\in {X_n}$, we have
\begin{align*}
  &\sup_{t\in[0,T]}||\cF_n(\vp(t,\cdot))-1||_{\wt\cM^\a}\\
  &\le||\vp_0-1||_{\wt\cM^\a}+\sup_{t\in[0,T]}\int_0^t||\cG_n(\vp(\tau,\cdot))-\g_2^n||_{\wt\cM_\a}d\tau\\
  &\le||\vp_0-1||_{\wt\cM^\a}+\int_0^T\g_\a^n||\vp(\tau)-1||_{\wt\cM^\a}+C||\vp(\tau)-1||_\b^2+C||\vp(\tau)-1||_\b d\tau\\
  &\le||\vp_0-1||_{\wt\cM^\a}+\g_\a^n T\sup_{t\in[0,T]}||\vp(t)-1||_{\wt\cM^\a}+CT(||\vp_0-1||_\b e^{\lam_\b^n T}+1)^2\\
  &<\infty.\\
\end{align*}
To show $\cF_n(\vp)\in {X_n}$, we also need to check that
\[
\begin{split}
||1-\cF_n(\vp)||_\b
&\leq||1-\vp_0||_\b e^{-\g_2^nt}+\int_0^t e^{-\g_2^n(t-\tau)}||\cG(\vp(\cdot,\tau))-\cG(1)||_\b d\tau\\
&\leq||1-\vp_0||_\b e^{-\g_2^nt}+\g_\b^n\int_0^t e^{-\g_2^n(t-\tau)}||\vp(\cdot,\tau)-1||_\b d\tau\\
&\leq||1-\vp_0||_\b e^{-\g_2^nt}+\g_\b^n\int_0^t e^{-\g_2^n(t-\tau)}e^{\lam_\b^n \tau}||\vp_0-1||_\b d\tau\\
&=e^{\lam_\b^n t}||1-\vp_0||_\b.
\end{split}
\]
%\begin{align}
%  ||\cF(\vp)-\cF(\p)||_{\tilde \cX_T^\a} &\leq \g_\a T||\vp-\p||_{\tilde \cX_T^\a}\quad
 % \forall \vp,\p\in \tilde \cX_T^\a.
%\end{align} 
Moreover, assume that
$\vp(t,\xi) \in C([0,T]; \wt \cM^\alpha)$. 
Since it follows from the definition of $\cF_n$  that
\begin{align*}
\cF_n(\vp)(t,\xi) -\cF_n(\vp)(s,\xi)&= \Big(\vp_0(\xi)-1\Big) \Big(e^{-\g_2^nt}- e^{-\g_2^ns}\Big)\\
&+\int_s^te^{-\g_2^n(t-\tau)}\Big( \cG_n(\vp(\tau))(\xi) - \g_2^n \Big)d\tau\\
&= I(t,s,\xi) + II(t,s,\xi),
\end{align*}
we have  $\cF_n(\vp)(t,\xi) \in C([0,T]; \wt \cM^\alpha)$.
 In fact, as for $\cK^\alpha$-norm, we have
\[\Big(\sup_{\xi}
\frac{|I|}{|\xi|^\alpha} \Big)(t,s) \le \Big|e^{-\g_2^nt}- e^{-\g_2^ns}\Big|\sup_{\xi} \frac{|\vp_0(\xi)-1|}{|\xi|^\alpha}  \rightarrow 0 \enskip \mbox{as} \enskip |t-s| \rightarrow 0,
\]
and 
\[
\Big(\sup_{\xi}\frac{|II| }{|\xi|^\alpha}\Big)(t,s)
\le \int_s^t \frac{\big| \cG_n(\vp(\tau))(\xi) - \g_2^n \big| }{|\xi|^\alpha}d\tau
\le |t-s| \g^n_\alpha \sup_{\tau \in [0,T]} \|1-\vp(\tau)\|_\alpha\,.
\]
The proof for $\wt \cM_\alpha$-norm is similar so that we omit the
detail. %Hence, $ \cF_n$ maps $X_n$ into itself.

Secondly, to prove $\cF_n$ is a contraction mapping, we apply Lemma \ref{lemm01} again to have that 
for any $\vp,\tilde\vp\in X_n$,
\[\begin{split}
&\sup_{t\in[0,T]}||\cF_n(\vp(t,\cdot))-\cF_n(\tilde\vp(t,\cdot))||_{\wt\cM^\a}\\
&\le\sup_{t\in[0,T]}\int_0^t||\cG_n(\vp(\tau,\cdot))-\cG_n(\vp(\tau,\cdot))||_{\wt\cM_\a}d\tau\\
&\le\int_0^T\g_\a^n||\vp(\tau)-\tilde\vp(\tau)||_{\wt\cM_\a}+C\max\{||1-\vp(\tau)||_\b,||1-\tilde\vp(\tau)||_\b\}\cdot ||\vp(\tau)-\tilde\vp(\tau)||_\b\\
&\quad+C||\vp(\tau)-\tilde\vp(\tau)||_\b+C||1-\tilde\vp(\tau)||_\b^{1-\ve}||\vp(\tau)-\tilde\vp(\tau)||_\b^\ve d\tau\\
&\le\g_\a^nT\sup_{t\in[0,T]}||\vp-\tilde\vp||_{\wt\cM_\a}+C\int_0^T(e^{\lam_\b^n \tau} ||1-\vp_0||_\b+1)d\tau\cdot \sup_{t\in[0,T]}||\vp-\tilde\vp||_\b\\
&\quad+CT\sup_{t\in[0,T]}||\vp-\tilde\vp||_\b+C||1-\vp_0||_\b^{1-\ve}\int_0^Te^{\lam_\b^n(1-\e)\tau}d\tau\sup_{t\in[0,T]}||\vp-\tilde\vp||_\b^\ve \\
&\le\g_\a^nT\sup_{t\in[0,T]}||\vp-\tilde\vp||_{\wt\cM_\a}+CT(e^{\lam_\b^n T} ||1-\vp_0||_\b+1)\cdot \sup_{t\in[0,T]}||\vp-\tilde\vp||_\b\\
&\quad+CT\sup_{t\in[0,T]}||\vp-\tilde\vp||_\b+CT||1-\vp_0||_\b^{1-\ve}e^{\lam_\b^n(1-\e)T}\sup_{t\in[0,T]}||\vp-\tilde\vp||_\b^\ve \\
&\le(\g_\a^n+C)(3+ ||1-\vp_0||_\b+||1-\vp_0||_\b^{1-\ve})e^{\lam_\b^n T}T\sup_{t\in[0,T]}dis_{\a,\b,\e}(\vp,\tilde\vp).
\end{split}\]
This together with the known estimate
$$\sup_{t\in[0,T]}||\cF_n(\vp(t,\cdot))-\cF_n(\tilde\vp(t,\cdot))||_\b\le \g_\b^nT\sup_{t\in[0,T]}||\vp(t,\cdot)-\tilde\vp(t,\cdot)||_\b,$$
we have
\begin{align}\label{contraction-ineq}
&\sup_{t\in[0,T]}dis_{\a,\b,\e}(\cF_n(\vp)-\cF_n(\tilde\vp))\\
&\quad\le[C_n(0)e^{\lam_\b^n T}T+\g_\b^nT+(\g_\b^n T)^{\e}]\sup_{t\in[0,T]}dis_{\a,\b,\e}(\vp,\tilde\vp),\nonumber
\end{align}
where $C_n(0)=(\g_\a^n+C)(3+ ||1-\vp_0||_\b+||1-\vp_0||_\b^{1-\ve})$. Therefore, $\cF_n:X_n\to X_n$ is a contraction mapping if we select $T>0$ small enough such that 
\begin{align}\label{contra-condi}
C_n(0)e^{\lam_\b^n T}T+\g_\b^n T+(\g_\b^nT)^{\e}<1.
\end{align}
And then it completes the proof of the lemma.
\end{proof}
Lemma \ref{contr-lemm} shows that there exists a unique solution $\vp(t,\xi)\in C([0,T_1],\wt\cM_\a)$ to the problem \eqref{c-p-fourier} under the cut-off assumption, for some $T_1>0$  depending 
on the initial datum $\vp_0$. To extend the solution to be global
in time, we can apply the above argument for the initial data $\vp(T_1,\xi)$, then one sufficient condition for $\cF_n$ to be a contraction mapping in the next 
time interval $t\in[T_1,T_2]$ is 
\begin{align*}
C_n(0)e^{\lam_\b^n (T_1+T_2)}T_2+\g_\b^nT_2+(\g_\b^nT_2)^{\e}<1,
\end{align*}
where we have used the fact that $C_n(T_1)\le C_n(0)e^{\lam_\b^n T_1}$. This 
process can be continued. Assume the length of $m$-th extension of time interval
is $T_m$, we can select the sequence $\{T_m\}_{m=1}^\infty$ as follows:
$$
C_n(0)e^{\lam_\b^n (T_1+T_2+\cdots+T_m)}T_m+\g_\b^nT_m+(\g_\b^nT_m)^{\e}=\frac12\,,\quad\forall m\in\bN.
$$
It is straightforward to check that $\sum_{m=1}^\infty T_m=\infty$.

\begin{rema}
Although we only obtain the solution in $X_n\subsetneq C([0,T],\wt\cM^\a)$, the fixed point is unique in $C([0,T],\wt\cM^\a)$. Assume $\p\in C([0,T],\wt\cM^\a)$ is a fixed point to the operator $\cF_n$. Then
$$
\p(t,\xi)=\vp_0(\xi)e^{-\g_2^nt}+\int_0^te^{-\g_2^n(t-\tau)}\cG_n(\p(\tau))(\xi)d\tau.
$$
Let $t=0$, we have $\p(0,\xi)=\vp_0$. Moreover,
\begin{align*}
e^{\g_2^nt}||\p(t,\cdot)-1||_\b&\le||\p_0-1||_\b+\int_0^te^{\g_2^n\tau}||\cG_n(\p(\tau))-\cG_n(1)||_\b d\tau\\
&\le||\p_0-1||_\b+\g_\b^n\int_0^te^{\g_2^n\tau}||\p(\tau)-1||_\b d\tau.
\end{align*}
Applying the Gronwall inequality yields
$$||\p(t,\cdot)-1||_\b\le e^{\lam_\b^n t}||\p_0-1||_\b,$$
that is , $\p\in X_n$.
\end{rema}

\subsection{Stability under cutoff assumption}
Let $\vp(t,\xi),\tilde\vp(t,\xi)\in C([0,T],\wt\cM_\a)$ be two solutions to the equation with cut-off cross section and with initial data $\vp_0,\tilde\vp_0$, respectively. 

Let $H(t,\xi)=\vp(t,\xi)-\tilde\vp(t,\xi)$, then it is known from the
previous works that 
\begin{equation}\label{k-b-stability}
||\vp(t,\xi)-\tilde\vp(t,\xi)||_\b\leq e^{\lam_\b^n t}||\vp_0-\tilde\vp_0||_\b,\quad\text{ for all }t>0.
\end{equation}
Starting from the equation, we can obtain
\begin{align}\label{3.2}
e^{\g_2^n t}\int\frac{|ReH(t,\xi)|}{|\xi|^{3+\a}}d\xi
\leq &\int\frac{|ReH(0,\xi)|}{|\xi|^{3+\a}}d\xi\\
&+\int_0^t e^{\g_2^n \tau}\int\frac{|Re\cG(\vp(\tau))-Re\cG(\tilde\vp(\tau))|}{|\xi|^{3+\a}}d\xi d\tau.\nonumber
\end{align}
By Lemma \ref{lemm01} and \eqref{k-b-stability}, we have
\[
\begin{split}
\int&\frac{|Re\cG(\vp(\tau))-Re\cG(\tilde\vp(\tau))|}{|\xi|^{3+\a}}d\xi\\
&\leq\g_\a^n \int \frac{|ReH(\tau)|} {|\xi|^{3+\a}}d\xi+ Ce^{2\lam_\b^n \tau}\max\{||1-\vp_0||_\b,||1-\tilde\vp_0||_\b\}\cdot ||\vp_0-\tilde\vp_0||_\b\\
&\quad+Ce^{\lam_\b^n \tau}||\vp_0-\tilde\vp_0||_\b+Ce^{\lam_\b^n \tau }||1-\tilde\vp_0||_\b^{1-\ve}||\vp_0-\tilde\vp_0||_\b^\ve.\\
\end{split}
\]
 To simplify the calculation, define the functions
\[
f(t)=e^{\g_2^n t}\int\frac{|ReH(t,\xi)|}{|\xi|^{3+\a}}d\xi,
\]

\[
g(t)=A\int_0^t e^{(\g_\b^n+\lam_\b^n) \tau }d\tau+B\int_0^t e^{\g_\b^n \tau }d\tau,
\]
where
\begin{align*}
A&=C\max\{||1-\vp_0||_\b,||1-\tilde\vp_0||_\b\}\cdot ||\vp_0-\tilde\vp_0||_\b,\\
B&=C(||\vp_0-\tilde\vp_0||_\b+||1-\tilde\vp_0||_\b^{1-\ve}||\vp_0-\tilde\vp_0||_\b^\ve).
\end{align*}
Then, (\ref{3.2}) becomes
$$
f(t)\leq f(0)+\g_\a^n\int_0^t f(\tau)d\tau+g(t).
$$
Applying the Gronwall inequality gives
\begin{align}\label{cutoff-stability}
\int\frac{|ReH(t,\xi)|}{|\xi|^{3+\a}}d\xi
\leq e^{\lam_\a^n t}\int\frac{|ReH(0,\xi)|}{|\xi|^{3+\a}}d\xi+\frac{e^{2\lam_\b^n t}-e^{\lam_\a^n t}}{2\lam_\b^n-\lam_\a^n}A+\frac{e^{\lam_\b^n t}-e^{\lam_\a^n t}}{\lam_\b^n-\lam_\a^n}B.
\end{align}
In particular, if $\tilde\vp=1$, we have
$$
||\vp(t,\xi)-1||_{\wt\cM^\a}
\leq e^{\lam_\a^n t}||\vp_0-1||_{\wt\cM^\a}+\frac{e^{2\lam_\b^n t}-e^{\lam_\a^n t}}{2\lam_\b^n-\lam_\a^n}A_0+\frac{e^{\lam_\b^n t}-e^{\lam_\a^n t}}{\lam_\b^n-\lam_\a^n}B_0,
$$
where $A_0=C||\vp_0-1||_\b^2\,,B_0=C||\vp_0-1||_\b$.

\subsection{Existence and stability  without cut-off assumption}\label{3-2}
Assume $b(\cdot)$ satisfies (\ref{index-sing}). Let $b_n=\min\{b,n\}$. For any $\vp_0\in \wt\cM^\a\subset\cK^\b,\a\in(\a_0,2)$ , we have a unique solution $\vp_n(t,\xi)\in C([0,\infty);\wt\cM^\a)$ to the Cauchy problem \eqref{c-p-fourier} with $b$ replaced by $b_n$.
From the $\cK^\b$-theory (\cite{Cannone-Karch-1,morimoto-12}), we know
$$
||\vp_n(t,\cdot)-1||_\b\le e^{\lam_\b^n t}||\vp_0-1||_\b
\le e^{\lam_\b t}||\vp_0-1||_\b,\mbox{ for all }\b\in(\a_0,\a],
$$ 
and it is proved that the sequence $\{\vp_n(t,\xi)\}_{n=1}^\infty$ is bounded and equicontinuous. Therefore,  by 
 Ascoli-Arzel\`a theorem, there exists a subsequence of solutions, denoted by $\{\vp_n\}$ again, which converges uniformly in every compact set of $[0,\infty)\times\bR^3$. Moreover, the limit function 
$$\vp(t,\xi)=\lim_{n\to\infty}\vp_n(t,\xi)$$
 is  a characteristic function in $\cK^\a$ and it is a solution to the problem \eqref{c-p-fourier} with the initial data $\vp(0,\xi)=\vp_0$.

To prove $\vp(t,\xi)\in \wt\cM^\a$, by letting $\tilde\vp=1$, we obtain from the stability estimate that for any $0<\d<1$
\begin{align*}
\int_{\d<|\xi|<\d^{-1}}\frac{|Re\vp-1|}{|\xi|^{3+\a}}d\xi
&=\lim_{n\to\infty}\int_{\d<|\xi|<\d^{-1}}\frac{|Re\vp_n-1|}{|\xi|^{3+\a}}d\xi
\\
&\leq\lim_{n\to\infty} e^{\lam_\a^n t}||\vp_0-1||_{\wt\cM^\a}+\frac{e^{2\lam_\b^n t}-e^{\lam_\a^n t}}{2\lam_\b^n-\lam_\a^n}A_0+\frac{e^{\lam_\b^n t}-e^{\lam_\a^n t}}{\lam_\b^n-\lam_\a^n}B_0\\
&\le e^{\lam_\a t}||\vp_0-1||_{\wt\cM^\a}+te^{2\lam_\b t}A_0+te^{\lam_\b t}B_0.
\end{align*}
Letting $\delta \rightarrow 0$, we obtain $\vp(t,\xi) \in \wt\cM^\a$ for each $t>0$. Similarly, the stability estimate \eqref{fourier-stability} follows from \eqref{cutoff-stability} by letting $n\to\infty$.

To complete the proof of Theorem  \ref{fourier-space},
we will show $\vp(t,\xi) \in C([0,\infty), \wt\cM^\a)$. For any $t,s>0$, we have
\begin{align*}
&Re\vp(t,\xi)-Re\vp(s,\xi)\\
&=\int_s^t\int_{\bS^2}
b(\cdot)\{(Re\vp(\tau,\xi^+)Re\vp(\tau,\xi^-)-Re\vp(\tau,\xi))d\s d\tau\\
&\quad-\int_s^t\int_{\bS^2}b(\cdot)Im\vp(\tau,\xi^+)Im\vp(\tau,\xi^-)\}d\s d\tau\\
&=I-II.\\
\end{align*}
For  $I$, by Proposition 3.3 in \cite{MWY-2014}, we obtain
\[
\int_{\bR^3}\frac{|I|}{|\xi|^{3+\a}}d\xi\lesssim |t-s|e^{\lam_\a\max\{s,t\}}||\vp_0-1||_{\wt\cM^\a}.
\]
For  $II$,
\begin{align*}
\int_{\bR^3}\frac{|II|}{|\xi|^{3+\a}}d\xi&\le\int_s^t \int_{\bS^2}b(\cdot)\left\{\int_{|\xi|\ge1}
\frac{||1-
	\vp(\tau,\xi^-)||_\b}{|\xi|^{3+\a-\b}}\sin^\b\frac\theta2 \right.\\
&\qquad\qquad\qquad\left.+\int_{|\xi|<1}\frac{||1-
	\vp(\tau,\xi^-)||_\b^2}{|\xi|^{3+\a-2\b}}\sin^\b\frac\theta2\cos^\b\frac\theta2\right\}d\s d\tau\\
&\lesssim (e^{\lam_\b\max\{s,t\}}||1-\vp_0||_\b
+1)^2|t-s|.\\
\end{align*}
Therefore,
$$
||\vp(t)-\vp(s)||_{\wt\cM^\a}\lesssim C(t,s)|t-s|,
\mbox{ for all }t,s>0,
$$
where $C(t,s)=e^{\lam_\b \max\{s,t\}}||\vp_0-1||_{\wt\cM^\a}+(e^{\lam_\b \max\{s,t\}}||1-\vp_0||_\b
+1)^2$.
And then it completes the proof of Theorem \ref{fourier-space}.

\subsection{Proof of Theorem \ref{existence-base-space}} 
We are now ready to complete the proof of Theorem \ref{existence-base-space}.
Assume $b$ satisfies \eqref{index-sing} for some $\a_0\in(0,2)$ and let  $\a\in[\a_0,2)$. If $F_0\in P_\a(\bR^3)$, then $\vp_0=\cF(F_0)\in\wt\cM^\a$. By Theorem \ref{fourier-space} and \eqref{equivalence-space}, there exists a unique measure valued solution $F_t\in C([0,\infty),P_\a(\bR^3))$ to the problem \eqref{bol}-\eqref{initial}. The continuity with respect to $t$ is in the following sense:
\begin{align}\label{p-a-continuity}
&\lim_{t \rightarrow t_0} \int \psi(v) dF(t,v) = \int \psi(v) dF(t_0,v),  %\enskip F_n=\cF^{-1}(\varphi_n), F =\cF^{-1}(\varphi)
% \in P_\a(\RR^3),
\mbox{ for any  $\psi \in C(\RR^3)$} \\
&\mbox{\qquad  \qquad \qquad
	satisfying
	the growth condition $|\psi(v)| \lesssim \la v \ra^\a$.}\notag
\end{align}
This is true because from \eqref{fourier-continuity-1}
and \eqref{fourier-continuity-2} , we have 
$$
dis_{\a,\b,\e}(\vp(t),\vp(t_0))\to0,\mbox{ as }t\to t_0.
$$
Then \eqref{p-a-continuity} follows from Theorem \ref{complete-1}, so that
 Theorem \ref{existence-base-space} holds.

%\noindent

\bigskip
\noindent
{\bf Acknowledgements:} The research of the first author was supported in part
by  Grant-in-Aid for Scientific Research No.25400160,
Japan Society for the Promotion of Science. The research of the third author was
supported in part by the General Research Fund of Hong Kong, CityU No. 11303614.
Authors would like to thank Yong-Kum Cho for his comments
on our previous work \cite{MWY-2014}.

%%%%%%%%%%%%%%%%%%%%%%%%%%%%%%%%%%%%%%%%%%%%%%%%%%%%%%%%%%%%%
%%%%%%%%%%%%%%%%%%%%%%%%%%%%%%%%%%%%%%%%%%%%%%%%%%%%%%%%%%%%%


\begin{thebibliography}{99}




\bibitem{ADVW} R. Alexandre, L. Desvillettes, C. Villani and B. Wennberg,
Entropy  dissipation and long-range interactions,  {\it Arch.
Rational Mech. Anal.} {\bf 152} (2000),  327-355.
%
%
%
%\bibitem{A-Sa} R. Alexandre and M. Elsafadi, {Littlewood Paley decomposition and regularity issues in Boltzmann homogeneous equations. I. Non cutoff and Maxwell cases}, {\it Math. Models Methods Appl. Sci.}
%{\bf 15 }%(6)
%(2005), 907-920.
%
%\bibitem{A-Sa2}R. Alexandre, M. Elsafadi, {Littlewood-Paley theory and regularity issues in
%Boltzmann homogeneous equations. II. Non cutoff case and non
%Maxwellian molecules}, { \it Discrete Contin. Dyn. Syst}. {\bf 24} (2009),
%1-11.
%
%
%
%
%




%
%
%\bibitem{AMUXY-KJM} R. Alexandre, Y. Morimoto, S. Ukai, C.-J. Xu
%and T. Yang, Smoothing effect of weak solutions
%for the spatially homogeneous Boltzmann equation without
%angular cutoff, {\it Kyoto J. Math.}
%{\bf 52} (2012), 433-463.
%
%\bibitem{arkeryd}
%L. Arkeryd, On the Boltzmann equation, {\it Arch. Rational Mech. Anal.,}
%{\bf 34} (1972), 1-34.
%
%\bibitem{bobylev-Cercignani}
%A. V. Bobylev and C. Cercignani, Self-similar solutions of the Boltzmann equation and their
%applications, {\it J. Statist. Phys.}{\bf 106}(2002), 1039-1071.


\bibitem{Cannone-Karch-1} M. Cannone and G. Karch,
{Infinite energy solutions to the homogeneous Boltzmann equation},
{\it Comm. Pure Appl. Math.} {\bf 63} (2010), 747-778.

%\bibitem{Cannone-Karch-2} M. Cannone and G. Karch,
%On self-similar solutions to the homogeneous Boltzmann equation, { \it Kinetic and Related Models}, {\bf 6} (2013), 801 - 808.


\bibitem{Carlen-Gabetta-Toscani} E. A. Carlen, E. Gabetta
 and G. Toscani, Propagation of smoothness and the rate of exponential convergence to equilibrium for a spatially homogeneous Maxwellian gas, {\it Comm. Math. Phys.},
{\bf 199} (1999) 521-546.


%\bibitem{carleman} T. Carleman, Sur la th\'eorie de l'\'equation
%int\'egrodiff\'erentielle de Boltzmann, {\it Acta Math.,} {\bf 60}
%(1933), 91-146.
%
%
%\bibitem{H-C} Y. Chen and L. He, { Smoothing estimates for Boltzmann equation with full-range interactions: Spatially homogeneous case},
%{\it  Arch. Rational Mech. Anal.} {\bf 201} (2011), 501-548.
%
%
%
%
%\bibitem{desv-wen1} L. Desvillettes and B. Wennberg, {Smoothness of the solution
% of the spatially homogeneous Boltzmann equation without cutoff}, {\it
% Comm. Partial Differential Equations} {\bf 29}(2004), 133--155.
%% \bibitem{dev-mouhot} L. Desvillettes and C. Mouhot, Stability and uniqueness for the spatially homogeneous Boltzmann equation with long-range interactions, {\it Arch. Ration. Mech. Anal.} {\bf 193} (2009), no. 2, 227-253.
%

\bibitem{Gabetta-Toscani-Wennberg}
E. Gabetta, G. Toscani and B. Wennberg, Metrics for probability distributions and the rend to equilibrium for solutions of the Boltzmann equation, {\it J. Statist. Phys}, {\bf 81}, 901-934.


\bibitem{Jacob}N. Jacob, Pseudo-differential operators and Markov processes. Vol 1:
Fourier analysis and semigroups. Imperial College Press, London, 2001.
%\bibitem{HMUY} Z.H. Huo, Y. Morimoto, S. Ukai and T. Yang, {Regularity of
%solutions for spatially homogeneous Boltzmann equation without
%Angular cutoff.} { \it Kinetic and Related Models}, {\bf 1} (2008),
%453-489.
%
%\bibitem{jacob}N. Jacob, Pseudo-differential operators and Markov processes. Vol 1:
%Fourier analysis and semigroups. Imperial College Press, London, 2001.
%
%\bibitem{lu-mouhot} X. Lu and C. Mouhot, On measure solutions of the
%Boltzmann equation, part I:
%Moment production and stability estimates, {\it
%Jour. Diff. Equa.} {\bf 252}(2012), 3305--3363.


%\bibitem{lu-wennberg}
%X. Lu and B. Wennberg, Solutions with increasing energy for the spatially
%homogeneous Boltzmann equation, {\it Nonlinear Anal. Real World Appl. }{\bf 3} (2002), 243-258.

\bibitem{mischler-wennberg}
S. Mischler and B. Wennberg, On the spatially homogeneous Boltzmann equation,
{\it Ann. Inst. H. Poincar\'e Anal. Non Lin\'eaire},
{\bf 16} (1999), 467-501.

\bibitem{morimoto-12}Y. Morimoto,
A remark on Cannone-Karch solutions
to the homogeneous Boltzmann equation for Maxwellian molecules,
{ \it Kinetic and Related Models}, {\bf 5} (2012), 551-561.

\bibitem{MUXY-DCDS}
Y. Morimoto, S. Ukai, C.-J. Xu and T. Yang, {Regularity of solutions to the
spatially homogeneous Boltzmann equation without angular cutoff},
{\it Discrete and Continuous Dynamical Systems -
 Series A} {\bf 24} (2009), 187--212.
%
%\bibitem{MU} Y. Morimoto and S. Ukai,
%Gevrey smoothing effect of solutions for  spatially homogeneous
%nonlinear Boltzmann  equation without angular cutoff,
%{ \it J. Pseudo-Differ. Oper. Appl.}, {\bf 1} (2010), 139-159.

\bibitem{MY} Y. Morimoto and T. Yang,
Smoothing effect of the homogeneous Boltzmann equation with measure valued initial datum, 
{\it Ann. Inst. H. Poincar\'e Anal. Non Lin\'eaire},
{\bf 32} (2015), 429-442.

\bibitem{MWY-2014}
Y. Morimoto, S. Wang and T. Yang, {A new characterization and global regularity 
	of infinite energy solutions to the homogeneous Boltzmann equation},
{\it J. Math. Pures Appl.} {\bf 103} (2015), 809--829.

%\bibitem{PT} A. Pulvirenti and G. Toscani, The theory of the nonlinear Boltzmann equation for Maxwell molecules in Fourier
%representation, {\it Ann. Mat. Pura Appl.} {\bf 171} (1996), 181-204.
%
%
%
%
%
%
%
%
%
%
%
%
%
%
%
%
%\bibitem{tanaka78}
%H. Tanaka, Probabilistic treatment of the Boltzmann equation of Maxwellian molecules, {\it Wahrsch. Verw. Geb.,} {\bf 46} (1978),
%67-105.



\bibitem{toscani-villani}G. Toscani and C. Villani, Probability metrics and uniqueness of the solution to the Boltzmann equations for Maxwell gas, {\it J. Statist. Phys.,} {\bf 94} (1999), 619-637.









\bibitem{villani}C. Villani, {On a new class of weak solutions to the spatially
homogeneous Boltzmann and Landau equations}, {\it Arch. Rational
 Mech. Anal.,} {\bf 143} (1998), 273--307.


\bibitem{villani2} C. Villani, A review of mathematical
topics in collisional kinetic theory. In: Friedlander S.,
Serre D. (ed.),
Handbook of Fluid Mathematical Fluid Dynamics, Elsevier Science   (2002).



\bibitem{villani3} C. Villani,
Topics in optimal transportation. Graduate Studies in Mathematics, {\bf 58}. American Mathematical Society, Providence, RI, (2003) .  %xvi+370 pp.
%\bibitem{villani3} C. Villani, private communication in  August,
%2008, Kyoto.

\end{thebibliography}
\end{document}